\documentclass[11pt,eqno]{article}
\usepackage{cite}
\usepackage{amssymb}

\usepackage{amsmath}
\allowdisplaybreaks[4]
\usepackage{array}
\usepackage{graphicx}
\usepackage{cases}
\usepackage{color}
\usepackage{multirow}
\usepackage{picinpar}
\usepackage{bm}

\usepackage[all]{xy}%

\definecolor{purple}{rgb}{0.5804,0.0000,0.8275}

\begin{document}
\today\\
\newcommand{\defi}{\stackrel{\Delta}{=}}
\newcommand{\qed}{\hphantom{.}\hfill $\Box$\medbreak}
\newcommand{\A}{{\cal A}}
\newcommand{\B}{{\cal B}}
\newcommand{\U}{{\cal U}}
\newcommand{\G}{{\cal G}}
\newcommand{\cZ}{{\cal Z}}
\newcommand{\proof}{\noindent{\bf Proof \ }}
\newcommand\one{\hbox{1\kern-2.4pt l }}
\newcommand{\Item}{\refstepcounter{Ictr}\item[(\theIctr)]}
\newcommand{\QQ}{\hphantom{MMMMMMM}}

\newtheorem{theorem}{Theorem}[section]
\newtheorem{proposition}{Proposition}[section]
\newtheorem{condition}{Condition}[section]
\newtheorem{lemma}{Lemma}[section]
\newtheorem{corollary}{Corollary}[section]
\newtheorem{definition}{Definition}[section]
\newtheorem{remark}{Remark}[section]
\newtheorem{assumption}{Assumption}[section]
\newtheorem{example}{Example}[section]
\newenvironment{cproof}
{\begin{proof}
 [Proof.]
 \vspace{-3.2\parsep}}
{\renewcommand{\qed}{\hfill $\Diamond$} \end{proof}}
\newcommand{\erhao}{\fontsize{21pt}{\baselineskip}\selectfont}
\newcommand{\xiaoerhao}{\fontsize{18pt}{\baselineskip}\selectfont}
\newcommand{\sanhao}{\fontsize{15.75pt}{\baselineskip}\selectfont}
\newcommand{\sihao}{\fontsize{14pt}{\baselineskip}\selectfont}
\newcommand{\xiaosihao}{\fontsize{12pt}{\baselineskip}\selectfont}
\newcommand{\wuhao}{\fontsize{10.5pt}{\baselineskip}\selectfont}
\newcommand{\xiaowuhao}{\fontsize{9pt}{\baselineskip}\selectfont}
\newcommand{\liuhao}{\fontsize{7.875pt}{\baselineskip}\selectfont}
\newcommand{\qihao}{\fontsize{5.25pt}{\baselineskip}\selectfont}

\makeatletter
\newcommand{\figcaption}{\def\@captype{figure}\caption}
\newcommand{\tabcaption}{\def\@captype{table}\caption}
\makeatother

\newcounter{Ictr}

\renewcommand{\theequation}{
\arabic{equation}}
\renewcommand{\thefootnote}{\fnsymbol{footnote}}

\def\A{\mathcal{A}}

\def\C{\mathcal{C}}
\def\F{\mathcal{F}}
\def\V{\mathcal{V}}
\def\K{\mathcal{K}}

\def\I{\mathcal{I}}

\def\Y{\mathcal{Y}}

\def\X{\mathcal{X}}

\def\J{\mathcal{J}}

\def\Q{\mathcal{O}}

\def\W{\mathcal{W}}

\def\S{\mathcal{S}}

\def\T{\mathcal{T}}

\def\L{\mathcal{L}}

\def\M{\mathcal{M}}

\def\N{\mathcal{N}}
\def\R{\mathbb{R}}
\def\H{\mathbb{H}}


\begin{center}
\topskip10mm
\LARGE{\bf Normal Cones Intersection Rule and Optimality Analysis for Low-Rank Matrix Optimization with Affine Manifolds}
\end{center}

\begin{center}
\renewcommand{\thefootnote}{\fnsymbol{footnote}}Xinrong Li$^{1}$,~Ziyan Luo$^{1}$
\footnote{
1. Department of Applied Mathematics,
Beijing Jiaotong University, Beijing 100044, P. R. China; X. Li (lixinrong0827@163.com), Z. Luo (starkeynature@hotmail.com).
}\\
{\small

}

\end{center}

\begin{abstract}
The low-rank matrix optimization with affine manifold (rank-MOA) aims to minimize a continuously differentiable function over a low-rank set intersecting with an affine manifold. This paper is devoted to the optimality analysis for rank-MOA.  As a cornerstone, the intersection rule of the Fr\'{e}chet normal cone to the feasible set of the rank-MOA is established under some mild linear independence assumptions. Aided with the resulting explicit formulae of the underlying normal cone, the so-called $F$-stationary point and the $\alpha$-stationary point of rank-MOA are investigated and the relationship with local/global minimizers are then revealed in terms of first-order optimality conditions. Furthermore, the second-order optimality analysis, including the necessary and the sufficient conditions, is proposed based on the second-order differentiation information of the model. All these results will enrich the theory of low-rank matrix optimization and give potential clues to designing efficient numerical algorithms for seeking low rank solutions. Meanwhile, two specific applications of the rank-MOA are discussed to illustrate our proposed optimality analysis.

\end{abstract}
\noindent\textbf{keywords}: Optimality conditions, low-rank set, affine manifold, normal cones, intersection rule


\vskip12pt

\numberwithin{equation}{section}

\section{Introduction}
As a reasonable and efficient characterization for dimensionality reduction and pattern recognition, the low-rankness has been witnessed and well explored for matrix data arising from a wide range of application problems. The resulting matrix optimization with embedded low rank matrix structures can be found in diverse areas such as system identification\cite{liu2009interior-point}, control\cite{SeogStructurally}, signal processing\cite{Berge1991A}, collaborative filtering\cite{gillis2011low-rank}, high-dimensional statistics\cite{Chen2015Fast,she2017robust},
finance\cite{42005Rank}, machine learning\cite{kobayashi2014low-rank,xing2002distance}, among others.

With affine manifold constraint as a prior, this paper focuses on the following low-rank matrix optimization problem
\begin{equation}\label{P}
\begin{aligned}
\min\limits_{X\in\mathbb{R}^{m\times n}}&~~f(X)\\
{\rm s.t.}&~~ \mathcal{A}(X)=b\\\nonumber
 & ~~\mathrm{rank}(X)\leq r,\nonumber
 \end{aligned}\tag{rank-MOA}
\end{equation}
where $f: \mathbb{R}^{m\times n}\rightarrow \mathbb{R}$ is a (twice) continuously differentiable function, $\mathcal{A}:\mathbb{R}^{m\times n}\rightarrow \mathbb{R}^l$ is a given linear map defined by
\begin{equation}\label{affine-def}
\mathcal{A}(X)=(\langle A^1, X \rangle, \ldots, \langle A^l, X \rangle )^{\top}
\end{equation} with $A^i\in\mathbb{R}^{m\times n}$ and $\langle A^i,X \rangle:=\sum_{k,j} A^i_{kj}X_{kj}$, $i=1,\ldots, l$, $b\in\mathbb{R}^l$  is a given vector,
${\rm rank}(X)$ denotes the rank of the matrix $X$,
and $r$ is a nonegative integer smaller than $n$, serving as the prescribed upper bound of the matrix rank. For convenience, we denote the involved affine manifold and the low-rank matrix set of \ref{P} by
  $$\mathcal{L}:=\{X\in\mathbb{R}^{m\times n}:\mathcal{A}(X)=b\},~~\mathcal{M}(r):=\{X\in\mathbb{R}^{m\times n}:\mathrm{rank}(X)\leq r\},$$
respectively, and the feasible region by $\mathcal{F}: =  \mathcal{L}\cap\mathcal{M}(r)$.

The NP-hardness, caused by the low-rank requirement in general low-rank matrix optimization \cite{fazel2002matrix}, inspires extensive study on relaxation theory and algorithms \cite{fazel2002matrix,  fazel2003, Fornasier2011Low, Mohan2012Iterative, Lai2013Improved}. Rather than the convex or nonconvex surrogates of the rank function in the aforementioned works, tackling the original rank function deduced low-rank matrix set will produce solution matrices of rank no more than any given upper bound. This is more appropriate and wanted, especially when the prescribed bound is a prior and is required as a hard constraint in specific application problems. To circumvent the nonconvexity and discontinuity of the rank function, the low-rank constraint is equivalently reformulated or transferred. We refer the interested readers to the papers \cite{burer2005local,JournLow, Gao2010STRUCTURED,delgado2016novel,zhou2018fast} and the references therein.  However, little research has been done in optimality theory for the original low-rank matrix optimization with only rank constraint, let alone the \ref{P}.

It is well-known that optimality conditions contribute a main content to optimization theory, and play a vital role in algorithm design in optimization methods. The fundamental variational tools include tangent cones and normal cones of the feasible region in the underlying optimization model. During the past few years, various notions of tangent and normal cones have been introduced to deal with the low-rank constraint directly. For instance, the Bouligand tangent cone to the low-rank set has been derived in \cite{schneider2015convergence}, the proximal and Mordukhovich normal cones to the low-rank set have been given in \cite{luke2013prox-regularity}, and the Clarke tangent cone and corresponding normal cone to the low-rank set have been presented in \cite{Li2018Social}. These explicit formulae, together with differentiation of objective functions, will then lead to optimality conditions for the corresponding low-rank matrix optimization.

When additional constraint set, namely $\Omega$, is embedded besides the low-rank set $\mathcal{M}(r)$, how to write out the explicit expression of the tangent and the normal cones of the underlying feasible region $\Omega \cap \mathcal{M}(r)$ will be essential for optimality analysis of the low-rank matrix optimization.
For $\Omega$ is spectral set\footnote{A set $\Omega$ is a spectral set if there exists a symmetric set $\mathcal{K}$  such that
$\Omega=\{X\in\mathbb{S}^n:  {\lambda}(X)\in \mathcal{K}\},$
where $\lambda$ denotes the spectral mapping.},
from \cite{lewis1996group,tam2017regularity}, one can obtain the desired normal cones though the intersection rules of the  normal cone for the problem with constraints of sparse and symmetric sets
\footnote{A set $\mathcal{K}\in\mathbb{R}^n$ is said to be symmetric if $Px\in\mathcal{K}$ for every $x\in\mathcal{K}$ and every $P\in\mathbb{P}^n$, where $\mathbb{P}^n$ denote the set of all $n\times n$ permutation matrices (Those matrices that have only one nonzero entry in every row and column, which is 1)} under the so-called R-LICQ in \cite[Corollary 2.9]{Pan2017Optimality}. Here, R-LICQ can be automatically satisfied for some typical choices for spectral sets.
For instance, Cason et al.\cite{Cason2013Iterative} discussed the Bouligand tangent cone and the corresponding normal cone when $\Omega$ is a unit sphere; Tam \cite{tam2017regularity} studied the Mordukhovich normal cone for the case $\Omega$ is the positive semi-definite cone, and Li et al. \cite{Li2019Social} considered the Fr\'{e}chet normal cones when $\Omega$ are closed unit Frobenius ball, the symmetric box or the spectrahedron.
However, for the case of affine manifold as considered in \ref{P},  the spectral structure of feasible set is destroyed. Therefore, the desired tangent cones and normal cones may not be accessible by intersection rules, unless some problem-tailored constraint qualifications (CQ) are proposed. There do exist very weak constraint qualifications to establish optimality conditions for constrained optimization, such as Guignards and Abadies contraint qualifications \cite{guignard1969generalized,Abadie1967}, but none of them are easy to verify since they can not bypass computing the tangent or normal cone of the constraint region. The challenge then turns to finding verifiable CQs that are applicable to \ref{P}.

It is also noteworthy that, in recent years, Riemannian manifold  optimization has been proved to be an effective approach to handle low-rank matrix optimization problems by applying tools for the fixed-rank matrix manifold. Manifold optimization theory and methods for such a type of low-rank matrix optimization then emerge \cite{Vandereycken2013,schneider2015convergence,2020On,Levin2021FindingSP}. However, these results can not be applied directly to \ref{P}, since the intrinsic fixed-rank matrix manifold is just a proper subset of the low-rank matrix set.

The aim of this paper is to study optimality conditions tailored for the \ref{P}. Main efforts will be focused on establishing the intersection rule of the Fr{\'e}chet normal cone under linear independence assumptions. 
The key idea is to construct an appropriate subset of the low-rank matrix set, with some easily tractable separability structure. Two types of stationary points for the \ref{P} are defined via the low-rank matrix projection and the Fr{\'e}chet normal cone.  First-order and the second-order optimality conditions in terms of these stationary points for the \ref{P} are then established under the aforementioned linear independence assumptions. For illustration purpose, we show how to apply our results to the problems of Hankel matrix approximation and low-rank representation on linear manifold. Notably, the exploration of optimality conditions for the \ref{P} not only makes up for the lack of optimality theory in structural low-rank optimization problems, but also enables many  optimization  algorithms to be applied to low-rank matrix optimization problems over sets of matrices which have to satisfy addition constraint. Currently, the main results in this paper do not cover the cases of  nonlinear equality and inequality constraints but some of the observations obtained alongside still provide ideas into these problems.

This paper is organized as follows. In Section 2, we review some related concepts and properties for normal cones, tangent cones and projections. In Section 3, we give the intersection rule of Fr\'{e}chet normal cone for the feasible set. In Section 4, we define two kinds of stationary points and investigate the first-order and second-order optimality conditions for the \ref{P}. In Section 5, we discuss some important applications of \ref{P} to illustrate our proposed optimality theory. Conclusions are made in Section 6.

{\sl Notation.}  Let $\mathbb{R}^{m\times n}$ be Euclidean space of the real $m \times n$ matrices equipped with the inner product $\langle X,Y\rangle=\sum_{i,j} X_{ij}Y_{ij}$ and the induced Frobenius norm $\|X\|_F$.
Denote by $\|X\|_2:=\sigma_1(X)$  the spectral norm of $X$ and $\sigma_1(X)$ is the largest singular value of $X$.
For any $X\in\mathbb{R}^{m\times n}$, we denote by $X_{ij}$ the $(i,j)$-th entry of $X$.  Let $J\subseteq\{1,\ldots, n\}$ be an index set. $|J|$ is the cardinality of $J$.
We use $X_J$ to denote the sub-matrix of $X$ that contains all columns indexed by $J$.
$\mathcal{O}^p$ is the set of all $p\times p$ orthogonal matrices, i.e., $\mathcal{O}^p=\{A\in\mathbb{R}^{p\times p}~|~A ^\top  A=AA ^\top =I_p\},$  where $I_p$  denotes the $p\times p$ identity matrix.
$O$ denotes the matrix with all components zero.
Let $\mathbb{R}^n$ be Euclidean space. For a vector $x\in\mathbb{R}^n$, let $\mathrm{Diag}(x)$ be an $n\times n$ diagonal matrix with diagonal entries $x_i$.

\section{Preliminaries}
This section presents several related concepts and properties regarded to normal cones, tangent cones and projections. Some of these properties are well known, some are less so, and all are basic. Most of them followed from the classical monograph \cite{Rockafellar2013Variational}.
\subsection{Normal cones and tangent cones}
A set $\mathcal{K}$ is called a cone, if $\gamma\mathcal{K}\subseteq \mathcal{K}$ holds for all $\gamma\geq 0$. The polar of the cone $\mathcal{K}$ is, denoted as $\mathcal{K}^\circ$, is defined by $\mathcal{K}^\circ=\{Y| \langle Y, X\rangle\leq 0,~ \forall X\in\mathcal{K}\}$.
If $\mathcal{K}_1$ and $\mathcal{K}_2$ are nonempty cones in $\mathbb{R}^{m\times n}$, we have
\begin{equation}\label{01}
(\mathcal{K}_1\cup \mathcal{K}_2)^\circ=(\mathcal{K}_1+\mathcal{K}_2)^\circ=\mathcal{K}_1^\circ\cap \mathcal{K}_2^\circ.
\end{equation}
Furthermore, if $\mathcal{K}_1$ and $\mathcal{K}_2$ are closed convex cones, then
\begin{equation}\label{02}
(\mathcal{K}_1\cap \mathcal{K}_2)^\circ=\mathcal{K}_1^\circ+ \mathcal{K}_2^\circ.
\end{equation}
For any given nonempty, closed set $\Omega \subseteq \mathbb{R}^{m\times n}$, and any $X\in \Omega$, the Bouligand tangent cone and its polar (also called the Fr\'{e}chet normal cone) to $\Omega$ at $X$, termed as $\mathrm{T}_\Omega^B(X)$ and $\mathrm{N}_\Omega^F(X)$, are defined by
\begin{eqnarray*}
\mathrm{T}_\Omega^B(X):&=&\left \{\Xi\in\mathbb{R}^{m\times n}: \begin{array}{l}\exists\{X^k\}\subseteq\Omega~\text{with}~X^k\rightarrow X;~\exists\{t_k\}~\text{with}\\  t_k\rightarrow0,~\mathrm{s.t.}~t_k^{-1}(X^k-X)\rightarrow\Xi,~\forall k\in\mathbb{N} \end{array}\right \},\\
\mathrm{N}_\Omega^F(X):&=&[\mathrm{T}_\Omega^B(X)]^\circ.
\end{eqnarray*}
Additionally, the Mordukhovich normal cone to $\Omega$ at $X$, termed as $\mathrm{N}_\Omega^M(X)$, is defined by
$$ \mathrm{N}_\Omega^M(X):=\limsup_{X'\xrightarrow{\Omega}X}\mathrm{N}_{\Omega}^F(X'),$$
where $X'\xrightarrow{\Omega}X$ means that $X'\in \Omega$ and $X'\rightarrow X$. It is seen that $\mathrm{N}_{\Omega}^F(X)\subseteq \mathrm{N}_\Omega^M(X)$.

Recall that $\Omega\subseteq \mathbb{R}^{m\times n}$ is locally closed in $\mathbb{R}^{m\times n}$ at $X$ if there exists a closed neighborhood $\mathcal{V}$ of $X$ such that $\Omega\cap \mathcal{V}$ is closed in $\mathbb{R}^{m\times n}$.

\begin{definition}(See \cite[Definition 6.4]{Rockafellar2013Variational})
The set~$\Omega$~being locally closed at $X\in\Omega$ and satisfying $\mathrm{N}_{\Omega}^F(X)=\mathrm{N}_\Omega^M(X)$ is called regular at $X$ in the sense of Clarke.
\end{definition}

The primary motivation for introducing regularity notions is to obtain equalities in calculus rules involving various constructs in nonsmooth analysis. Particularly, the smooth manifolds are Clarke regular, and the general tangent and normal cones defined above reduce to the tangent and normal spaces
\begin{eqnarray*}
\mathrm{T}_\Omega(X):&=&\mathrm{T}_\Omega^B(X),\\
\mathrm{N}_\Omega(X):&=&\mathrm{N}_{\Omega}^F(X)=\mathrm{N}_\Omega^M(X).
\end{eqnarray*}

Related results on tangent and normal cones are displayed for the sequent analysis.

{\tt Affine manifold.} The tangent space to $\mathcal{L}$ at $X\in\mathcal{L}$ is given by
$$\mathrm{T}_\mathcal{L}(X)=\text{ker} \mathcal{A}:=\{\Xi\in\mathbb{R}^{m\times n}: \langle A^i, \Xi\rangle =0, i=1,\ldots, l\}.$$
The normal space to $\mathcal{L}$ at $X$ is then its orthogonal complement, namely
$$\mathrm{N}_\mathcal{L}(X)=\text{range}\mathcal{A}^*:=\left\{\sum_{i=1}^l y_iA^i: y\in\mathbb{R}^l\right\}.$$

{\tt Fixed-rank manifold.} Denote the set of matrices of rank $s$ in $\mathbb{R}^{m\times n}$ by $$\mathcal{M}^{s}:=\{X\in\mathbb{R}^{m\times n}: \mathrm{rank}(X)=s\}.$$
It is well known that $\mathcal{M}^{s}$ is a smooth manifold (see \cite{Helmke1995Critical}), which is called the fixed-rank manifold.

For any given matrix $X \in \mathcal{M}^{s}$, with its singular value decomposition (SVD)
\begin{equation}\label{SVD}
X =U \Sigma{V} ^\top = [U_{\Gamma}~ U_{\Gamma_m^\perp} ]
     \left[
      \begin{array}{cc}
       {\Sigma}_{\Gamma }& O \\
       O & O \\
     \end{array}
   \right][ V_{\Gamma } ~ {V}_{\Gamma_n^\perp} ] ^\top,
 \end{equation}
where $U\in\mathcal{O}^m$ and $V\in\mathcal{O}^n$, $\Gamma\subseteq \{1, \ldots, \min(m,n)\}$ is the index set for nonzero singular values with $|\Gamma|=s$, $\Sigma_{\Gamma}\in \mathbb{R}^{s\times s}$ is the submatrix of the diagonal matrix $\Sigma$ indexed by $\Gamma$, $U_{\Gamma_m^\perp}$ and ${V}_{\Gamma_n^\perp}$ are the orthogonal complements of $U_{\Gamma}$ and $V_{\Gamma}$, with
 $\Gamma_m^\bot = \{1,\ldots, m\}\setminus \Gamma$, and $\Gamma_n^\bot = \{1,\ldots, n\}\setminus \Gamma$.
The corresponding tangent and normal cones (spaces) have the following explicit formulae,
\begin{eqnarray*}
\mathrm{T}_{\mathcal{M}^{s}}(X )&=&\left\{H\in\mathbb{R}^{m\times n}: U_{\Gamma_m^\perp}^\top HV_{\Gamma_n^\perp}=O \right\},\\
\mathrm{N}_{\mathcal{M}^{s}}(X )&=&\left\{{U}_{\Gamma_m^\perp}D{V}_{\Gamma_n^\perp}^\top\in \mathbb{R}^{m\times n}: D\in\mathbb{R}^{(m-s)\times (n-s)}\right\}.
\end{eqnarray*}
A program on $\mathcal{M}^s$ can be viewed as a Riemannian optimization on $\mathbb{R}^{m\times n}$ with the Riemannian metric defined by
$g_X(A,B)=\langle A, B \rangle$, where $X\in\mathcal{M}^s$ and $A, B\in\mathrm{T}_{\mathcal{M}^s}(X)$.

{\tt Low-rank set.}
With the aid of the above expressions of tangent and normal spaces to the rank-fixed matrix set $\mathcal{M}^s$, the explicit formulae for the Bouligand tangent cone and the Fr\'{e}chet normal cone, and the Mordukhovich normal cone to the low-rank matrix set $\mathcal{M}(r)$, have been characterized in \cite[Theorem 3.2]{schneider2015convergence} and \cite[Proposition 3.6]{luke2013prox-regularity}. The results are summarized as below.
\begin{lemma}\label{NBF}
For any $X \in\mathcal{M}(r)$ of rank $s$, we have
 \begin{eqnarray*}
\mathrm{T}^B_{\mathcal{M}(r)}(X )&=&\mathrm{T}_{\mathcal{M}^{s}}(X )+\{H\in \mathrm{N}_{\mathcal{M}^{s}}(X ): \mathrm{rank}(H)\leq r-s\},\\
\mathrm{N}^F_{\mathcal{M}(r)}(X )&=&\begin{cases}
\mathrm{N}_{\mathcal{M}^{s}}(X ),&s=r,\\
\{O\},&s<r.
\end{cases}\\
\mathrm{N}^M_{\mathcal{M}(r)}(X)&=&\{W\in \mathrm{N}_{\mathcal{M}^{s}}(X): \mathrm{rank}(W)\leq \min(m,n)-r\}.
\end{eqnarray*}
\end{lemma}

\subsection{Projections}
Given a nonempty and closed set $\Omega\subset\mathbb{R}^{m\times n}$,  the projection of an element $X\in\mathbb{R}^{m\times n}$ onto $\Omega$ is defined as $\Pi_\Omega(Z):=\arg\min_{Y\in \Omega}\|Y-X\|_F$ which is always nonempty, and is a singleton if $\Omega$ is convex in addition.

{\tt Projection onto low-rank set.} Given $Z\in {\mathbb{R}}^{m\times n}$ of rank $s$ with nonzero singular values $\sigma_1(Z)$, $\ldots$, $\sigma_s(Z)$, where $\sigma_i(Z)$ are the $i$-th largest singular values of $Z$. The projection of $Z$ onto $\mathcal{M}(r)$ is given by
\begin{equation*}
\Pi_{\mathcal{M}(r)}(Z)=\left\{U\mathrm{Diag}(\sigma_1(Z),\ldots,\sigma_r(Z),0,\ldots,0)V^\top  \Big| (U,V)\in \mathcal{O}^{m,n}(Z)\right\},
\end{equation*}
where $$\mathcal{O}^{m,n}(Z): = \left\{(U,V)\in \mathcal{O}^m \times \mathcal{O}^n: Z = U \left[
                                                                                                 \begin{array}{cc}
                                                                                                    \mathrm{Diag}(\sigma_1(Z),\ldots,\sigma_s(Z)) & O \\
                                                                                                   O & O \\
                                                                                                 \end{array}
                                                                                               \right]
 V^\top \right\}.$$
We remark that this projection may not be unique when $\sigma_r(Z)=\sigma_{r+1}(Z)$.

{\tt Projection onto the tangent and the normal space of $\mathcal{M}^s$.} Given $X\in \mathcal{M}^s$ with its SVD as stated in \eqref{SVD}, the projection onto $\mathrm{T}_{\mathcal{M}^s}(X)$ and $\mathrm{N}_{\mathcal{M}^s}(X)$ take the forms of
  $$\Pi_{\mathrm{T}_{\mathcal{M}^s}(X)}(Z)= \Pi_{U_{\Gamma}}Z\Pi_{V_{\Gamma}}+\Pi_{U_{\Gamma}}Z \Pi^\perp_{V_{\Gamma}}+ \Pi^\perp_{U_{\Gamma}}Z\Pi_{V_{\Gamma}}, ~~\forall Z\in\mathbb{R}^{m\times n},$$
  and
   $$\Pi_{\mathrm{N}_{\mathcal{M}^s}(X)}(Z)=  \Pi^\perp_{U_{\Gamma}}Z \Pi^\perp_{V_{\Gamma}}, ~~\forall Z\in\mathbb{R}^{m\times n},$$
  where $\Pi_{U_{\Gamma}}=U_{\Gamma}U_{\Gamma}^{\top}$ and $\Pi^\perp_{U_{\Gamma}}=U_{\Gamma_m^\perp}U_{\Gamma_m^\perp}^{\top}$.

Suppose that $f:\mathcal{M}^s\rightarrow \mathbb{R}$ is twice continuously differentiable. From \cite{Vandereycken2013}, the Riemannian gradient of $f$ at $X\in\mathcal{M}^s$
is defined as
$$\text{grad}f(X):=\Pi_{\mathrm{T}_{\mathcal{M}^s}(X)}(\nabla f(X)),$$
 and the Riemannian Hessian of $f$ at $X$ is the linear map $\text{Hess}f(X):\mathrm{T}_{\mathcal{M}^s}(X)\rightarrow\mathrm{T}_{\mathcal{M}^s}(X)$ defined as
  \begin{eqnarray*}
  \text{Hess}f(X)[\Xi]:=\Pi_{\mathrm{T}_{\mathcal{M}^s}(X)}(\nabla^2f(X )[\Xi])+\mathcal{H}(\nabla f(X)).
   \end{eqnarray*}
  Here \begin{equation}\label{H}
   \mathcal{H}(\nabla f(X)):=\Pi^\perp_{U_{\Gamma}}\nabla f(X)Q\Sigma_{\Gamma}^{-1}V_{\Gamma}^{\top}\Pi_{V_{\Gamma}}
   + \Pi_{U_{\Gamma}}U_{\Gamma}\Sigma_{\Gamma}^{-1}P^T\nabla f(X)\Pi^\perp_{V_{\Gamma}},
   \end{equation}
   with  $Q=\Pi^\perp_{V_{\Gamma}} \Xi^{\top} U_{\Gamma}$ and $P=\Pi^\perp_{U_{\Gamma}} \Xi V_{\Gamma}$. Moreover, for any $\Xi\in\mathrm{T}_{\mathcal{M}^s}(X)$, we get
   \begin{eqnarray}\label{Hess-terms}
   \text{Hess}f(X)[\Xi,\Xi]&=&\langle\Pi_{\mathrm{T}_{\mathcal{M}^s}(X)}(\nabla^2f(X )[\Xi]), \Xi\rangle\nonumber\\
    &+&
  \langle \Pi^\perp_{U_{\Gamma}}\nabla f(X )\Pi^\perp_{V_{\Gamma}}, \Xi V_{\Gamma}\Sigma_{\Gamma}^{-1}U_{\Gamma}^{\top} \Xi \rangle+ \langle \Pi^\perp_{U_{\Gamma}}\nabla f(X )\Pi^\perp_{V_{\Gamma}}, \Xi V_{\Gamma}\Sigma_{\Gamma}^{-1}U_{\Gamma}^{\top} \Xi \rangle\nonumber\\
  &=&\nabla^2f(X)[\Xi,\Xi]+2\langle\Pi_{\mathrm{N}_{\mathcal{M}^s}(X)}(\nabla f(X)),\Xi X^{\dag}\Xi \rangle,
    \end{eqnarray}
where $X^{\dag}=V_{\Gamma}\Sigma_{\Gamma}^{-1}U_{\Gamma}^{\top}$ is the Moore-Penrose inverse of $X$.  The term $\nabla^2f(X)[\Xi,\Xi]$ in \eqref{Hess-terms} contains second-order information about $f$ along the tangent space $\mathrm{T}_{\mathcal{M}^s}(X)$ but only first-order information on $\mathcal{M}^s$. The second term in \eqref{Hess-terms} involves second-order information about $\mathcal{M}^s$ but only first-order information about $f$ along the normal space(see \cite{Vandereycken2013} in detail).

\section{Normal cones intersection rule}
Calculus rules of the normal cones of sets play a crucial role in optimality conditions for a nonsmooth mathematical program. As the feasible set of \ref{P} is an intersection of the low-rank matrix set $\mathcal{M}(r)$ and the smooth manifolds $\mathcal{L}$, we will discuss the intersection rule of the normal cones to such an intersection set in this section.
 The next lemma in  \cite[Theorem 6.42]{Rockafellar2013Variational} will be used in the sequel.
\begin{lemma}[Tangents and normals to unions and intersections]\label{TNUI}
 Let $X\in \bigcup^k_{i=1}\Omega_i$ for closed sets $\Omega_i\in\mathbb{R}^{m\times n}$. It holds that
 \begin{equation}\label{03}
\mathrm{T}^B_{\bigcup^k_{i=1}\Omega_i}(X)=\bigcup^k_{i=1}\mathrm{T}^B_{\Omega_i}(X).
 \end{equation}
Let $X\in \Omega_1\cap\Omega_2$. It holds that
\begin{equation}\label{04}
\mathrm{T}^B_{\Omega_1\cap\Omega_2}(X)\subseteq \mathrm{T}^B_{\Omega_1}(X)\cap \mathrm{T}^B_{\Omega_2}(X),~~~~\mathrm{N}^F_{\Omega_1\cap\Omega_2}(X)\supseteq \mathrm{N}^F_{\Omega_1}(X)+\mathrm{N}^F_{\Omega_2}(X).
\end{equation}
Under basic qualification condition (BQ) $\mathrm{N}^M_{\Omega_1}(X)\cap(-\mathrm{N}^M_{\Omega_2}(X))=\{0\}$, we also has
\begin{equation}\label{05}
\mathrm{N}^M_{\Omega_1\cap\Omega_2}(X)\subseteq \mathrm{N}^M_{\Omega_1}(X)+\mathrm{N}^M_{\Omega_2}(X).
\end{equation}
If in addition $\Omega_1$ and $\Omega_2$ are regular at $X$, then $\Omega_1\cap \Omega_2$ is regular at $X$ and
\begin{equation}\label{06}
 \mathrm{T}^B_{\Omega_1\cap\Omega_2}(X)=\mathrm{T}^B_{\Omega_1}(X)\cap \mathrm{T}^B_{\Omega_2}(X),~~~~\mathrm{N}^F_{\Omega_1\cap\Omega_2}(X)=\mathrm{N}^F_{\Omega_1}(X)+\mathrm{N}^F_{\Omega_2}(X).
\end{equation}
\end{lemma}
Particular, if $\Omega_1$ and $\Omega_2$ are two smooth manifolds, then the BQ is equivalent to the transversality in  \cite{lewis2008alternating}. Thus, $\Omega_1\cap \Omega_2$ is a  smooth manifold at $X$ and there holds \eqref{06}.

Since the lack of regularity, the low-rank set cannot ensure the equality form in the calculus rules of the normal cones in \eqref{06} of Lemma \ref{TNUI}. Therefore, we choose the union of a finite number of subspaces as a subset of low-rank set and replace the normal cone of $\mathcal{M}(r)$ with the normal cone of a low-rank subset at rank-deficient point. To construct the low-rank subset, we first define some nontation. For $X\in \mathcal{L}\cap \mathcal{M}(r)$ with its SVD as in \eqref{SVD}, we denote $\mathcal{J}:=\{J\subseteq\{1,\ldots, \min(m,n)\}:|J|=r, \Gamma\subseteq J\}$ and introduce the following subset of the low-rank matrix set with respect to any given $X\in \mathcal{M}(r)$ together with the matrices $U$ and $V$ in \eqref{SVD}
\begin{equation}\label{SX}
\mathcal{M}_{(X,U,V)}(r):=\bigcup_{J\in\mathcal{J}}\mathcal{M}_{(X,U,V)}(J)\subseteq \mathcal{M}(r),
\end{equation}
 where
\begin{equation}\label{SJX}
\mathcal{M}_{(X,U,V)}(J):=\left\{ \begin{array}{ll}
           \left\{U_{J} B V^\top: B\in \mathbb{R}^{r\times n}\right\}, & \text{if}~m\leq n,\\
     \left\{U B V_J^\top: B\in \mathbb{R}^{m\times r}\right\},&\text{if} ~m\geq n.\\
             \end{array}
\right.
\end{equation}
is a subspace associated with $(X,U,V)$. For simplicity, we use $\mathcal{M}_X(r)$ and $\mathcal{M}_X(J)$ to briefly denote $\mathcal{M}_{(X,U,V)}(r)$ and $\mathcal{M}_{(X,U,V)}(J)$, respectively. Without loss of generality, we assume that  $m\geq n$ in the remainder of this paper. Particularly, if $J=\Gamma$, we have
$\mathcal{M}_X(r)=\mathcal{M}_X(\Gamma) = \left\{U B V_{\Gamma}^\top: B\in \mathbb{R}^{m\times r}\right\}.$
\begin{lemma}\label{NNN}
Let $X \in \mathcal{M}(r)$ be a rank $s$ matrix with its SVD as in \eqref{SVD}, and $\mathcal{M}_X(J)$ and $\mathcal{M}_X(r)$ be defined as in \eqref{SJX} and \eqref{SX}. The following statements hold.
\begin{itemize}
\item[(i)] For any $J\in \mathcal{J}$,
\begin{equation}\label{TNspace}
\mathrm{N}_{\mathcal{M}_X(J)}(X )
=\{W\in\mathbb{R}^{m\times n}:U^\top W V_J=O\}.\end{equation}
\item[(ii)]
\begin{equation}\label{TNspace1}
\mathrm{N}^F_{\mathcal{M}_X(r)}(X)=\left\{ \begin{array}{ll}
           \left\{
  W\in\mathbb{R}^{m\times n}:U^\top W V_\Gamma=O \right\}, & \text{if}~s=r\\
     \mathrm{N}^F_{\mathcal{M}(r)}(X )=\{O\},&\text{if} ~s< r.\\
             \end{array}
\right.
\end{equation}
\end{itemize}
\end{lemma}
\begin{proof}
The first part follows readily from the definition of the subspace $\mathcal{M}_X(J)$. Note that
 \begin{eqnarray}\label{J0}
\mathrm{N}^F_{\mathcal{M}_X(r)}(X )&=&\left(\mathrm{T}^B_{\mathcal{M}_X(r)}(X )\right)^\circ
=\left(\mathrm{T}^B_{\bigcup \limits_{J\in\mathcal{J} }\mathcal{M}_X(J)}(X )\right)^\circ=\left(\bigcup \limits_{J\in\mathcal{J} }\mathrm{T}_{\mathcal{M}_X(J)}(X )\right)^\circ\nonumber\\
&=&\bigcap \limits_{J\in\mathcal{J} }\mathrm{N}_{\mathcal{M}_X(J)}(X )\nonumber\\
&=&\bigcap \limits_{J\in\mathcal{J} }\{W\in\mathbb{R}^{m\times n}: U^\top WV_J=O\}\nonumber\\
&=& \{W\in\mathbb{R}^{m\times n}: U^\top WV_J=O,~\forall J\in\mathcal{J}\}.
\end{eqnarray}
Rewrite $\mathcal{J}=\{J_1,\ldots, J_{t_0}\}$ with $t_0=\left(
                                                               \begin{array}{c}
                                                                 n-s \\
                                                                 r-s \\
                                                               \end{array}
                                                             \right)$
 being the combinatorial number. Then $\bigcup_{i=1}^{t_0}J_i=\{1,\ldots,n\}$. If $s=r$, then $\mathcal{J}=\{\Gamma\}$, and hence $\mathrm{N}^F_{\mathcal{M}_X(r)}(X)=\{W\in\mathbb{R}^{m\times n}:U^\top WV_\Gamma=O \}$. If $s<r$, for any $J_i\in\mathcal{J}$ and $W\in\mathrm{N}^F_{\mathcal{M}_X(r)}(X )$, we have
 $U^\top WV_{J_i}=O$, which indicates that $U^\top WV=O$. Thus, we get that $\mathrm{N}^F_{\mathcal{M}_X(r)}(X )=\{O\}$. This completes the proof.
\end{proof}

Let $X\in \mathbb{R}^{m\times n}$ with its SVD as in \eqref{SVD}. For any given matrices $A^1$, $\ldots$, $A^l\in\mathbb{R}^{m\times n}$, denote
\begin{equation}\label{Ti}
 T^i_X=
     \left[
      \begin{array}{cc}
      U_{\Gamma} ^\top A^iV_{\Gamma}&   U_{\Gamma} ^\top A^iV_{{\Gamma}_n^\bot} \\
       U_{{\Gamma }_m^\bot} ^\top  A^iV_{\Gamma }& 0 \\
     \end{array}
   \right],~~~~~
R^i_X= U^\top A^iV_{\Gamma}
\end{equation}
for $i=1,\ldots,l$.
Introduce the following two assumptions.

\begin{assumption}\label{assumption 1}
The matrices $T^i_X$, $i=1,\ldots, l$, are linearly independent.
\end{assumption}

\begin{assumption}\label{assumption 2}
The matrices $R^i_X$, $i=1,\ldots, l$, are linearly independent.
\end{assumption}

It is worth mentioning that Assumption \ref{assumption 1} is equivalent to Assumption \ref{assumption 2} in symmetric matrices space and they are uniformly called the primal nondegeneracy condition in \cite[Definition 5]{Alizadeh1997Complementarity} in the context of semidefinite programming. However, Assumption \ref{assumption 2} is a stronger variant of Assumption \ref{assumption 1} in $\mathbb{R}^{m\times n}$.  Let $X$ be a feasible point of the problem \ref{P} with $\text{rank}(X)=s$. By the discussion of \cite[Section 5, Page 480]{bonnans2000perturbation}, we have that Assumption \ref{assumption 1} can happen only if  $l\leq mn-(m-s)(n-s)$. Similarly, a necessary condition for Assumption \ref{assumption 2} holding is $l\leq ms$. Based on these two assumptions, we have the following BQ holds.

\begin{proposition}\label{Assu}
 For any $X\in\mathcal{L}\cap\mathcal{M}(r)$ with its SVD as in \eqref{SVD}, and any index set $J$ satisfying $\Gamma \subseteq J$, we have
\begin{itemize}
\item[(i)]  If  Assumption \ref{assumption 1} holds at $X$, then $\mathrm{N}^M_{\mathcal{M}(r)}(X )\cap \mathrm{N}_{\mathcal{L}}(X )=\{O\}$;
\item[ (ii)]  If  Assumption \ref{assumption 2} holds at $X$, then $\mathrm{N}_{\mathcal{M}_X(J)}(X )\cap \mathrm{N}_{\mathcal{L}}(X )=\{O\}$.
\end{itemize}
\end{proposition}
\begin{proof} 
(i) By virtue of Theorem 6 in \cite {Alizadeh1997Complementarity}, Assumption \ref{assumption 1} holds at $X $ if and only if $\mathrm{N}_{\mathcal{M}^s}(X )\cap \mathrm{N}_{\mathcal{L}}(X )=\{O\}$.  The desired assertion in (i) then follows from the fact $\mathrm{N}^M_{\mathcal{M}(r)}(X )\subseteq \mathrm{N}_{\mathcal{M}^s}(X )$.

(ii) Assume on the contrary that there exists a nonzero matrix $$W\in \mathrm{N}_{\mathcal{M}_X(J)}(X )\cap \mathrm{N}_{\mathcal{L}}(X ),$$ that is, there exist $t^i\in\mathbb{R}$ ($i=1,\ldots, l$) not all zero such that $\sum_{i=1}^l t^iA^i \in  \mathrm{N}_{\mathcal{M}_X(J)}(X )$. From \eqref{TNspace},  we get $U ^\top  \Sigma t^iA^iV_{J}=O$, which implies  that $$ U^\top  \Sigma_{i=1}^l t^iA^iV_{\Gamma }=O.$$ Thus
\begin{eqnarray*}
\Sigma_{i=1}^l t^iR_X^i=
     U^\top  \Sigma_{i=1}^l t^iA^iV_{\Gamma} =O
   \end{eqnarray*}
   for~$i=1,\ldots,l$, which contradicts to the linear independency of $R_X^i$'s in Assumption \ref{assumption 2}. Thus, we have $\mathrm{N}_{\mathcal{M}_X(J)}(X )\cap \mathrm{N}_{\mathcal{L}}(X )=\{O\}$. This completes the proof.
 \end{proof}

\begin{lemma}\label{TBeq}
Let $X \in\mathcal{L}\cap\mathcal{M}(r)$ with its SVD as in \eqref{SVD}. If Assumption \ref{assumption 2} holds at $X$, then
\begin{equation}\label{Nequ}
\mathrm{T}^B_{\mathcal{L}\cap \mathcal{M}_X(r)}(X )=\mathrm{T}_{\mathcal{L}}(X )\cap \mathrm{T}^B_{\mathcal{M}_X(r)}(X ).
\end{equation}
\end{lemma}
\begin{proof}
Note that $\mathcal{L}$ and $\mathcal{M}_X(J)$ are regular at $X$. Since Assumption \ref{assumption 2} holds at $X$, from Proposition \ref{Assu} (ii) and \eqref{06}, we obtain that

\begin{eqnarray}\label{TJ}\mathrm{T}_{\mathcal{L}\cap\mathcal{M}_X(J)}(X )=\mathrm{T}_{\mathcal{L}}(X )\cap \mathrm{T}_{\mathcal{M}_X(J)}(X ).\end{eqnarray}
This together with \eqref{03} yields 
\begin{eqnarray}\label{T}
\mathrm{T}^B_{\mathcal{L}\cap\mathcal{M}_X(r)}(X )&=&\bigcup \limits_{J\in\mathcal{J} }\mathrm{T}_{\mathcal{L}\cap\mathcal{M}_X(J)}(X )
=\bigcup \limits_{J\in\mathcal{J} }\left(\mathrm{T}_{\mathcal{L}}(X )\cap \mathrm{T}_{\mathcal{M}_X(J)}(X )\right)\\\nonumber
&=&\mathrm{T}_{\mathcal{L}}(X )\cap \left(\bigcup \limits_{J\in\mathcal{J} }\mathrm{T}_{\mathcal{M}_X(J)}(X )\right)\\\nonumber
&=&\mathrm{T}_{\mathcal{L}}(X )\cap \mathrm{T}^B_{\bigcup \limits_{J\in\mathcal{J} }\mathcal{M}_X(J)}(X )\\\nonumber
&=&\mathrm{T}_{\mathcal{L}}(X )\cap \mathrm{T}^B_{\mathcal{M}_X(r)}(X ).
\end{eqnarray}
This completes the proof.
\end{proof}

Based on the above result, we state and prove the intersection rule of the Fr\'{e}chet normal cone to $\mathcal{L}\cap\mathcal{M}(r)$.

\begin{theorem}\label{NFeq}
Let $X \in\mathcal{L}\cap\mathcal{M}(r)$ with $s=:\mathrm{rank}(X )$.
\begin{itemize}
\item[(i)] If $s=r$ and Assumption \ref{assumption 1} holds at $X$, then
\begin{equation}\label{Nequ1}
\mathrm{N}^F_{\mathcal{L}\cap\mathcal{M}(r)}(X )=\mathrm{N}_{\mathcal{L}}(X )+\mathrm{N}^F_{\mathcal{M}(r)}(X ).
\end{equation}
\item[(ii)] If $s<r$ and Assumption \ref{assumption 2} holds at $X$, then
\begin{equation}\label{Nequ2}
\mathrm{N}^F_{\mathcal{L}\cap\mathcal{M}(r)}(X)=\mathrm{N}_{\mathcal{L}}(X )+\mathrm{N}^F_{\mathcal{M}(r)}(X)=\mathrm{N}_{\mathcal{L}}(X).
\end{equation}
\end{itemize}
\end{theorem}
\begin{proof}
(i) If $s=r$, it is known from Lemma \ref{NBF} that $\mathrm{N}^F_{\mathcal{M}(r)}(X ) = \mathrm{N}_{\mathcal{M}^s}(X)= \mathrm{N}^M_{\mathcal{M}(r)}(X )$. Thus, in this case, $\mathcal{M}(r)$ is regular at $X$. Together with the regularity of the convex set $\mathcal{L}$ at $X$, we obtain that
$\mathcal{L}\cap\mathcal{M}(r)$ is also regular at $X$, i.e., $\mathrm{N}^F_{\mathcal{L}\cap\mathcal{M}(r)}(X) = \mathrm{N}^M_{\mathcal{L}\cap\mathcal{M}(r)}(X)$. By utilizing (i) of Proposition \ref{Assu}, Assumption \ref{assumption 1} ensures that
$\mathrm{N}^M_{\mathcal{M}(r)}(X)\cap \left(-\mathrm{N}_{\mathcal{L}}(X)\right) = \mathrm{N}^M_{\mathcal{M}(r)}(X)\cap \mathrm{N}_{\mathcal{L}}(X) =\{O\}$. Thus, from \eqref{05} in Lemma \ref{TNUI}, $\mathrm{N}^F_{\mathcal{L}\cap\mathcal{M}(r)}(X)\subseteq \mathrm{N}^F_{\mathcal{M}(r)}(X) + \mathrm{N}_{\mathcal{L}}(X)$. Combining with the second inclusion in \eqref{04}, the desired assertion is obtained.

(ii) The second equality follows readily from \eqref{TNspace1}. For the remaining equality, by virtue of the second inclusion in \eqref{04}, it suffices to show
$$\mathrm{N}^F_{\mathcal{L}\cap\mathcal{M}_X(r)}(X)\subseteq \mathrm{N}_{\mathcal{L}}(X )+\mathrm{N}^F_{\mathcal{M}_X(r)}(X ).$$ From Lemma \ref{TBeq}, \eqref{01}  and \eqref{02}, we obtain that
\begin{eqnarray}\label{N}
\mathrm{N}^F_{\mathcal{L}\cap\mathcal{M}_X(r)}(X )&=&\left(\mathrm{T}^B_{\mathcal{L}\cap\mathcal{M}_X(r)}(X)\right)^\circ
=\left(\bigcup \limits_{J\in\mathcal{J} }\mathrm{T}_{\mathcal{L}}(X )\cap \mathrm{T}_{\mathcal{M}_X(J)}(X )\right)^\circ\\\nonumber
&=&\bigcap \limits_{J\in\mathcal{J} }\left(\mathrm{T}_{\mathcal{L}}(X )\cap \mathrm{T}_{\mathcal{M}_X(J)}(X ) \right)^\circ\\\nonumber
&=&\bigcap \limits_{J\in\mathcal{J} }\left(\mathrm{N}_{\mathcal{L}}(X )+ \mathrm{N}_{\mathcal{M}_X(J)}(X ) \right).
\end{eqnarray}
For any $H\in \mathrm{N}^F_{\mathcal{L}\cap\mathcal{M}_X(r)}(X )$, we have $H\in \mathrm{N}_{\mathcal{L}}(X )+ \mathrm{N}_{\mathcal{M}_{X}(J)}(X )$, for any $J\in\mathcal{J}$, that is, there exist $t^i(J)\in\mathbb{R}$, $i=1,\ldots, l$,  $W(J)\in \mathrm{N}_{\mathcal{M}_X(J)}(X ),$
  such that
  \begin{equation}\label{HW}
  H=\sum_{i=1}^l t^i(J)A^i+W(J), ~\forall~J\in\mathcal{J}.
   \end{equation}
Note that for each $J\in \mathcal{J}$, it holds that $U^\top W(J)V_{\Gamma}=O$ from \eqref{TNspace}.
  Pre- and post-multiplying both sides of the equation \eqref{HW} by $U_{\Gamma}^\top$ and $V_{\Gamma}$, respectively, we obtain
$$U^\top HV_{\Gamma}=\sum_{i=1}^l t^i(J)U_{\Gamma}^\top A^i V_{\Gamma}, ~\forall~J\in\mathcal{J}.$$
For any distinct index set $J_0\in\mathcal{J}$, we can also get that
  $$U^\top HV_{\Gamma}=\sum_{i=1}^l t^i(J_0)U_{\Gamma}^\top A^i V_{\Gamma},  ~\forall~J_0(\neq J)\in\mathcal{J}.$$
  Invoking the linear independency in Assumption \ref{assumption 2}, we have
  $$t^i(J)=t^i(J_0)=:t^i, ~\forall~J,J_0\in\mathcal{J}.$$
  This implies that
  $$H-\sum_{i=1}^l t^iA^i\in\mathrm{N}_{\mathcal{M}_X(J)}(X),~\forall~J\in\mathcal{J}$$
  i.e.,
  $$H-\sum_{i=1}^l t^iA^i\in \bigcap \limits_{J\in\mathcal{J}} \mathrm{N}_{\mathcal{M}_X(J)}(X)= \mathrm{N}^F_{\mathcal{M}_X(r)}(X).$$
  Thus, $H\in \mathrm{N}_{\mathcal{L}}(X)+\mathrm{N}^F_{\mathcal{M}_X(r)}(X)$, which indicates that $\mathrm{N}^F_{\mathcal{L}\cap\mathcal{M}(r)}(X)\subseteq \mathrm{N}_{\mathcal{L}}(X )+\mathrm{N}^F_{\mathcal{M}_X(r)}(X )=\mathrm{N}_{\mathcal{L}}(X )+\mathrm{N}^F_{\mathcal{M}(r)}(X)$.
This yields the assertions in \eqref{Nequ2}.
\end{proof}

\begin{remark}
Theorem \ref{NFeq} has generalized the decomposition property in \cite[Corollary 2.10]{Pan2017Optimality} from vectors to matrices.  Due to the disadvantage that the low-rank matrix set can not be decomposed into a union of a finite number of subspaces, the proof much more complicated than that for vectors while the sparse vector set does.

Specifically, by treating any vector $x$ in $\mathbb{R}^n$ as a diagonal matrix ${\text{Diag}}(x)$ in $\mathbb{R}^{n\times n}$, the involved orthogonal matrices $U$ and $V$ in the SVD for the latter diagonal matrix are both reduced to $I_n$. Thus, the so-called R-LICQ for $\widetilde{\mathcal{F}}:=\widetilde{\mathcal{L}}\cap S$
with  $$S:=\{x\in \mathbb{R}^n: \|x\|_0\leq r\},~\widetilde{\mathcal{L}}:=\{x\in \mathbb{R}^n: a_i^\top x = b_i, i=1,\ldots, l\}$$
at $x$ that introduced in \cite[Definition 2.4]{Pan2017Optimality} is equivalent to both Assumptions \ref{assumption 1} and \ref{assumption 2} for $$\mathcal{F}:=\left\{{\text{Diag}}(x)\in \mathbb{R}^{n\times n}: \langle \text{Diag}(a_i), {\text{Diag}}(x)\rangle = b_i, i = 1,\ldots, l, ~\text{rank(Diag}(x))\leq r\right\}$$
at ${\text{Diag}}(x)$. Meanwhile, the sparse set $S$ corresponds exactly to the low-rank matrix set $\mathcal{M}_{{\text{Diag}}(x)}(r)$.
\end{remark}

\section{Optimality Conditions}\label{sec:Opti}
The optimality analysis, including the first-order and the second-order optimality conditions for the \ref{P}, is proposed in this section, which will provide necessary theoretical fundamentals for handling such a nonconvex discontinuous matrix programming problem.

\subsection{Stationarity}
We begin by the introduction of two types of stationary points for the \ref{P}.
For any $X\in \mathcal{M}(r)$ and any $y\in \mathbb{R}^l$, define the Lagrangian function associated with the \ref{P} by
 \begin{equation}\label{LF}
L(X;y)=f(X)+\sum_{i=1}^ly_i[\langle A^i,X\rangle-b_i].
 \end{equation}

\begin{definition}\label{def-stat}
 Suppose $\alpha>0$ and $X \in\mathcal{M}(r)$.
\begin{itemize}
 \item[(i)] $X $ is called an $F$-stationary point of \ref{P} if there exists a vector $y \in\mathbb{R}^l$ such that
\begin{equation}\label{F-sta}
\left\{ \begin{array}{lr}
            \mathcal{A}(X )=b, &  \\
 -\nabla_XL(X ;y )\in \mathrm{N}^{F}_{\mathcal{M}(r)}(X )&\\
\end{array}
\right.
\end{equation}
\item[(ii)] $X $ is called an $\alpha$-stationary point of \ref{P} if there exists a vector $y \in\mathbb{R}^l$ such that
\begin{equation}\label{alpha-stationary-def}
\left\{ \begin{array}{lr}
            \mathcal{A}(X )=b, &  \\
             X \in\Pi_{\mathcal{M}(r)}(X -\alpha\nabla_XL(X ;y )). &
             \end{array}
\right.
\end{equation}

\end{itemize}
\end{definition}

The relationship between the above $F$- and $\alpha$- stationary point for the \ref{P} are discussed in the following proposition.

\begin{proposition}\label{F-M-alpha}
For any given $X \in\mathcal{L}\cap\mathcal{M}(r)$ with $s:=\text{rank}(X)$, $y\in\mathbb{R}^l$, and $\alpha>0$, denote
$$\beta=\left\{ \begin{array}{ll}
          \dfrac{{\sigma}_r(X)}{\|\nabla_{X}L(X ;y))\|_2}, & \text{if}~ \nabla_{X}L(X ;y)\neq 0,\\
            \infty, &\text{otherwise}.
             \end{array}
\right.$$
Consider the statements: (a) $X$ is an $F$-stationary point of \ref{P}; (b) $X$ is an $\alpha$-stationary point of \ref{P}; 
We have
\begin{description}
\item[(i)] $(b)\Rightarrow (a)$;
\item[(ii)] if $s=r$ and $\alpha\in(0,\beta]$, then $(a)\Rightarrow(b)$;
\item[(iii)] if $s < r$, then $(a)\Rightarrow(b)$.
\end{description}
\end{proposition}
\begin{proof}
 By mimicking the proof of Theorem 2 in \cite{Li2018Social}, we can obtain that $X$ is an $\alpha$-stationary point of \ref{P} if and only if
\begin{equation}\label{a}
\nabla_{X}L(X ;y)=\begin{cases}
U_{\Gamma_m^\bot}D V_{\Gamma_n^\bot}^\top,~\mathrm{with}~ \|\nabla_{X}L(X ;y))\|_2\leq\frac{1}{\alpha}{\sigma}_r(X),&\textrm{if~}s=r,\\
O,&\textrm{if~} s<r,\\
\end{cases}
\end{equation}
where $D\in \mathbb{R}^{(m-r)\times (n-r)}$. Together with the expressions of $\mathrm{N}^{F}_{\mathcal{M}(r)}(X)$
as presented in Lemma \ref{NBF}, we can obtain all the desired assertions.
\end{proof}

\subsection{First-order optimality}
The first-order optimality conditions in terms of the $F$-stationary point are stated as follows.
\begin{theorem}\label{F-min}
Let $X \in\mathcal{L}\cap\mathcal{M}(r)$ of rank $s$.
\begin{itemize}
\item[(i)]  Suppose that  $X $ is a local minimizer of \ref{P}. If $s=r$ and Assumption \ref{assumption 1} holds at $X$ or $s<r$ and Assumption \ref{assumption 2} holds at $X$, then $X$ is an $F$-stationary point of \ref{P}.
\item[(ii)] Suppose that $f$ is a convex function and $X$ is an $F$-stationary point of \ref{P}. If $s=r$, then $X$ is a global minimizer of \ref{P} restricted on $\mathcal{M}_X(\Gamma)$; If $s<r$, then $X$ is a global minimizer of \ref{P}.
\end{itemize}
\end{theorem}
\begin{proof}
 (i) If $X$ is a local  minimizer of \ref{P}, it follows from the generalized Fermat's theorem and Theorem \ref{NFeq} that
\begin{equation}\label{Fermat} -\nabla f(X )\in \mathrm{N}^F_{\mathcal{L}\cap\mathcal{M}(r)}(X) = \mathrm{N}_{\mathcal{L}}(X) + \mathrm{N}^F_{\mathcal{M}(r)}(X),\end{equation}
if (a) $s=r$ and Assumption \ref{assumption 1} holds at $X$, or (b) $s<r$, Assumption \ref{assumption 2} holds at $X$. Together with the fact $$\mathrm{N}_{\mathcal{L}}(X) = \left\{\sum\limits_{i=1}^l y^i A^i: y^i\in \mathbb{R}, i=1,\ldots, l\right\},$$
 \eqref{Fermat} indicates that there exists $y\in \mathbb{R}^l$, such that $-\nabla_X L(X;y)\in \mathrm{N}^F_{\mathcal{M}(r)}(X)$. This yields the necessary optimality conditions for \ref{P} as stated in (i).

(ii) If $s<r$, it follows from \eqref{F-sta} that there exists $y\in\mathbb{R}^l$ such that $$\nabla_X L(X;y) =O~~\text{and}~~L(X;y) = f(X).$$ For any feasible solution $Y$ of \ref{P}, it yields that
$$ f(Y) = L(Y;y) \geq L(X;y)+\langle \nabla_X L(X;y), Y-X\rangle = L(X;y) = f(X),$$
where the inequality follows from the convexity of the function $L(\cdot;y)$ due to the convexity of $f$. Thus we conclude that $X$ is a global solution of \ref{P}. If $s = r$, then \eqref{F-sta} implies that there exist $y\in \mathbb{R}^l$ and $D\in \mathbb{R}^{(m-r)\times (n-r)}$ such that
\begin{equation}\label{La} \nabla_X L(X;y) = U_{\Gamma_m^\bot}DV^\top_{\Gamma_n^\bot}.
\end{equation} For any $Y\in \mathcal{M}_X(\Gamma)\cap \mathcal{L}$, we can find some matrix $B\in \mathbb{R}^{m \times r}$ such that $Y = U B V_\Gamma^\top$. Thus,
$$ f(Y) = L(Y;y) \geq L(X;y)+\langle \nabla_X L(X;y), Y-X\rangle = L(X;y) = f(X),$$
where the inequality follows from the convexity of $L(\cdot;y)$, and the second equality is from \eqref{La} and $Y-X = UBV_\Gamma^\top- U_{\Gamma}\Sigma(X)V_\Gamma^\top $. This completes the proof.
\end{proof}

Note that the $F$-stationarity condition for the case of $s<r$ is reduced to the classic KKT condition for the problem 
\begin{equation}\label{new-p}\min\{f(X): {\cal A}(X) =b\}.\end{equation} However, as declared in Theorem \ref{F-min} (i), at a local minimizer $X$ of \ref{P} with rank$(X)<r$, the inactive low-rank constraint can not be ruled out to achieve the stationarity of $X$ with respect to \eqref{new-p}, unless some additional constraint qualification is satisfied. This is caused by the discontinuity of the rank function.  The following example indicates that even for the unique global minimizer of \ref{P} with rank strictly less than $r$, it is not an $F$-stationary point.

\begin{example}\label{Ex4.4}
Consider the problem
\begin{equation}\label{LAF}
  \begin{aligned}
\min\limits_{X\in\mathbb{R}^{3\times 3}}&~~ \langle e_2e_2^\top, X\rangle\\\
{\rm s.t.}&~~ \langle e_1e_1^\top-e_2e_2^\top, X\rangle=0,\\
          &~~ \langle e_3e_3^\top, X\rangle=1,\\
          &~~ \langle e_ie_j^\top, X\rangle=0,~~~i\neq j, i,j = 1,2,3,\\
          &~~ \mathrm{rank}(X)\leq 2.
  \end{aligned}
\end{equation}
It is easy to check that $\overline{X}=e_3e_3^\top$ is the unique global minimizer. Note that for any $y=(y_1,\ldots, y_8)^\top\in\mathbb{R}^{8}$,
$$\nabla_XL(\overline{X};y)=\left[
 \begin{array}{ccc}
  y_1& y_2&y_3\\
  y_4& 1-y_1&y_5\\
  y_6& y_7&y_8\\
    \end{array}
 \right]\neq O.$$
 Thus, $\overline{X}$ is not an $F$-stationary point of  \eqref{LAF}. One can easily verify that Assumption \ref{assumption 2} fails at $\overline{X}$, since $R^1_{\overline{X}} = O$ by choosing $\overline{U}=\overline{V}=\left[
 \begin{array}{ccc}
   e_3& e_2&e_1\\
    \end{array}
 \right]$ to diagonalize $\overline{X}$. This also indicates that Assumption \ref{assumption 2} can not be removed in Theorem  \ref{F-min} (i).  
\end{example}

The first-order optimality conditions via $\alpha$-stationarity are proposed as below, by utilizing Theorem \ref{F-min} and Proposition \ref{F-M-alpha}.

\begin{theorem}\label{alpha-min}
Let $X \in\mathcal{L}\cap\mathcal{M}(r)$ of rank $s$.
\begin{itemize}
\item[(i)] Suppose that $X $ is a local minimizer of \ref{P}. If $s=r$ and  Assumption \ref{assumption 1} holds at $X$, then there exists $y \in\mathbb{R}^l$ such that, for any $0<\alpha\leq\beta$, $X$ is an $\alpha$-stationary point of \ref{P}.
If $s<r$, Assumption \ref{assumption 2} holds at $X$, then there exists $y \in\mathbb{R}^l$ such that, for any $\alpha>0$, $X$ is an $\alpha$-stationary point of \ref{P}.
\item[(ii)] Suppose that $f$ is convex and $X$ is an $\alpha$-stationary point of \ref{P}. If $s=r$, then $X$ is a global minimizer of \ref{P} restricted on $\mathcal{M}_X(\Gamma)$; If $s<r$, then $X$ is a global minimizer of \ref{P}. Furthermore, if $f$ is strong convex  with modulus $l_f>0$, for $\alpha\geq\frac{1}{l_f}$, then $X$ is the unique global minimizer of \ref{P}.
\end{itemize}
\end{theorem}
\begin{proof} Following from the relationship as declared in Proposition \ref{F-M-alpha} and the optimality as proposed in Theorem \ref{F-min}, we only need to show the ``furthermore" part in (ii).
It is easy to verify that $L(X;y)$ is strongly convex in $X$ with the same modulus $l_f>0$ of $f$. Then, for any $Y\in\mathcal{L}\cap\mathcal{M}(r)$, we have
$$L(Y;y)-L(X;y)\geq\langle\nabla_X L(X;y),Y-X\rangle+\frac{l_f}{2}\|Y-X\|_F^2.$$
Since $X$ is an $\alpha$-stationary point with $y$ when $\alpha\geq\frac{1}{l_f}$, from Definition \ref{def-stat} (ii), we have
$$X\in\Pi_{\mathcal{M}(r)}(X-\alpha\nabla_ XL(X;y)), ~~~\mathcal{A}(X)=b.$$
This indicates that for any $X\neq Y\in \mathcal{L}\cap\mathcal{M}(r)$,
$$\|X-(X-\alpha\nabla_ XL(X;y))\|^2\leq\|Y-(X-\alpha\nabla_ XL(X;y))\|^2.$$
Simple manipulation leads to
$$\langle \nabla_ XL(X;y), Y-X\rangle\geq-\frac{1}{2\alpha}\|Y-X\|^2.$$
Thus, for any $X\neq Y\in \mathcal{L}\cap\mathcal{M}(r)$,
\begin{eqnarray*}
 f(Y)-f(X)&=&L(Y;y)-L(X;y)\\\nonumber
 &\geq&\langle\nabla_X  L(X;y),Y-X\rangle+\frac{l_f}{2}\|Y-X\|_F^2 \\\nonumber
   &\geq&\frac{1}{2}(l_f-\frac{1}{\alpha})\|X-Y\|_F^2\geq 0.
\end{eqnarray*}
This shows that $X$ is the unique global minimizer of \ref{P}.
\end{proof}

One might argue that the necessary optimality via the $F$-stationarity and the $\alpha$-stationary as stated in Theorem \ref{F-min} (i) and Theorem \ref{alpha-min} (i) are too restrictive and the required assumptions are too strong, especially for the case of $s<r$, just like what happened in Example \ref{Ex4.4}. A possible remedy is to employ the Mordukhovich normal cone (the outer limit of the Fr\'{e}chet normal cone) instead of the original Fr\'{e}chet normal cone, and define the so-called $M$-stationary point (see, e.g., \cite{Pan2017Optimality,Li2018Social}). 
%
%
The first-order optimality conditions in terms of the $M$-stationary point are stated as follows.
\begin{corollary}\label{M-min}
Let $X \in\mathcal{L}\cap\mathcal{M}(r)$ of rank $s$.
\begin{itemize}
\item[(i)]  Suppose that  $X$ is a local minimizer of \ref{P}. If Assumption \ref{assumption 1} holds at $X$, then $X$ is an $M$-stationary point of \ref{P}.
\item[(ii)] Suppose that $f$ is a convex function and $X$ is an $M$-stationary point of \ref{P}. Then $X$ is a global minimizer of \ref{P} restricted on $\mathcal{M}_X(\Gamma)$.
\end{itemize}
\end{corollary}
\begin{proof}
(i) Under the BQ condition for \ref{P}, i.e., Assumption \ref{assumption 1}, one can obtain that any local minimizer is an $M$-stationary point by employing the inclusion property \eqref{05}.

(ii) Since $X$ is an $M$-stationary point of \ref{P}, then there exist $y\in \mathbb{R}^l$ and $D\in \mathbb{R}^{(m-r)\times (n-r)}$ with $\mathrm{rank}(D)\leq n-r$ such that
\begin{equation*} \nabla_X L(X;y) = U_{\Gamma_m^\bot}DV^\top_{\Gamma_n^\bot}.
\end{equation*}
Using the same proof as Theorem \ref{F-min} (ii), we can derive $X$ is a global minimizer of \ref{P} restricted on $\mathcal{M}_X(\Gamma)$.
This completes the proof.
\end{proof}


Note that the so called $M$-stationary point is weaker than the classical $M$-stationarity defined by $$ -\nabla f(X)\in \mathrm{N}^{M}_{\mathcal{L}\cap\mathcal{M}(r)}(X )$$ since the  Mordukhovich normal cones intersection rule may not hold. Moreover, since the $M$-stationarity is much weaker than both of the $F$- and $\alpha$-stationary points, it may has less power of ruling out the non-optimal feasible solutions. The next example shows such a case.
\begin{example}
Consider the problem
\begin{equation}\label{Tr}
  \begin{aligned}
\min\limits_{X\in\mathbb{R}^{4\times 4}}&~~f(X):=\frac{1}{2}\|H-X\|_F^2\\
{\rm s.t.}&~~ \langle I_4, X\rangle=2,\\
 &~~ \mathrm{rank}(X)\leq 3,
  \end{aligned}
\end{equation}
where $H=\left[
 \begin{array}{cccc}
     0& 0& -e_3& 0 \\
    \end{array}
 \right]$. Consider a feasible solution $X_1=\left[
 \begin{array}{cccc}
    e_1& e_2&0& 0 \\
    \end{array}
 \right]$.
Apparently, we can choose $U=V=I_4$ to diagonalize $X_1$. Using Lemma \ref{NBF}, one has
\begin{equation}\label{X1-FM}
\mathrm{N}^F_{\mathcal{M}(3)}(X_1)=\{O\},  ~\mathrm{N}^M_{\mathcal{M}(3)}(X_1)=\left\{\left[
                                                                                         \begin{array}{cc}
                                                                                           O & O \\
                                                                                           O & H \\
                                                                                         \end{array}
                                                                                       \right]\in {\mathbb{R}}^{4\times 4}: H\in {\mathbb{R}}^{2\times 2}, {\rm rank}(H)\leq 1
\right\}.
\end{equation}
Note that for any $y\in \mathbb{R}$,
$$\nabla_XL(X_1;y)=X_1-H+yI_4=\left[
 \begin{array}{cccc}
   (1+y)e_1& (1+y)e_2&(1+y)e_3& ye_4\\
    \end{array}
 \right]\neq O$$
and $$-\nabla_XL(X_1;-1)=-e_4e_4^\top \in \mathrm{N}^M_{\mathcal{M}(3)}(X_1).$$
Thus, $X_1$ is not an $F$-stationary point, but an $M$-stationary point of problem \eqref{Tr}.
Similarly, we can also verify that all the following three feasible solutions are $M$-stationary points
$$X_2=\left[
 \begin{array}{cccc}
    0& e_2&0& e_4\\
    \end{array}
 \right],~~X_3=\left[
 \begin{array}{cccc}
    e_1& 0&0& e_4\\
    \end{array}
 \right],~~X_4=\left[\begin{array}{cccc}
       \frac{2}{3}e_1&  \frac{2}{3}e_2&0&  \frac{2}{3}e_4\\
     \end{array}
  \right],$$
among which only $X_4$ is an $F$-stationary point. Moreover, for any scalar $\alpha\in \left(0, 2\right)$,
\begin{eqnarray}
X_4-\alpha\nabla_XL\left(X_4;-\dfrac{2}{3}\right)&=& \left[\begin{array}{cccc}
    \frac{2}{3}e_1& \frac{2}{3}e_2& -\dfrac{\alpha}{3}e_3&  \frac{2}{3}e_4 \nonumber\\
    \end{array}\right]
\end{eqnarray}
and $$X_4=\Pi_{\mathcal{M}(3)}\left(X_4-\alpha\nabla_XL\left(X_4;-\dfrac{2}{3}\right)\right).$$
Thus, $X_4$ is also an $\alpha$-stationary point for any $\alpha\in (0,2)$. Note that $f$ is strongly convex with modulus $l_f = 1$. By virtue of
the first-order sufficient condition in Theorem \ref{alpha-min} (ii), we conclude that $X_4$ is the unique global minimizer of \eqref{Tr}.
\end{example}

\subsection{Second-order optimality}
Next, we study the second-order necessary and sufficient optimality conditions for the problem \ref{P}.

\begin{theorem}\label{necessary}
Suppose $f$ is twice continuously differentiable on $\mathbb{R}^{m\times n}$. If $X \in\mathcal{L}\cap\mathcal{M}(r)$ with the SVD as in \eqref{SVD} is a local minimizer of \ref{P}, then we have the following statements.
\begin{itemize}
\item[(i)] If $s = r$ and Assumption \ref{assumption 1}  holds at $X $, then
\begin{equation}\label{so-necessary1}
\nabla^2f(X )[\Xi,\Xi]-2\langle \nabla_X L(X;y),\Xi X^{\dag}\Xi \rangle\geq 0, ~~~\forall\Xi \in \mathrm{T}_{\mathcal{L}}(X )\cap\mathrm{T}^B_{\mathcal{M}
(r)}(X ).
\end{equation}
\item[(ii)] If $s<r$ and Assumption \ref{assumption 2}  holds at $X $, then\begin{equation}\label{so-necessary2} \nabla^2f(X )[\Xi,\Xi]\geq 0, ~~~\forall\Xi \in \mathrm{T}^B_{{\mathcal{L}}\cap\mathcal{M}(r)}(X )
\end{equation}
\end{itemize}
where $\nabla^2f(X )$ is the Hessian of $f$ at $X$ on $\mathbb{R}^{m\times n}$.
\end{theorem}
\begin{proof}

\noindent (i)   Consider the case $s=r$. In this case, problem \ref{P} can be thought as the equality constrained minimization with a Riemannian manifold $\mathcal{M}^r$.
  If $s =r$, then Theorem \ref{F-min} (i) implies that $X$ is an $F$-stationary point of the \ref{P} and hence there exists $y\in \mathbb{R}^l$ such that $-\nabla_X L(X;y)\in\mathrm{N}_{\mathcal{M}^r}(X)$,
which implies $\text{grad} L(X;y)=\Pi_{\mathrm{T}_{\mathcal{M}^r}(X)}(\nabla_X L(X;y))=O$. It is easy to verify that $\Pi_{\mathrm{T}_{\mathcal{M}^r}(X)}(A^i)=UT_X^iV^{\top}$ and  hence Assumption \ref{assumption 1} is identical to the LICQ in Theorem 4.2 in \cite{2014Optimality}. According to Theorem 4.2 in \cite{2014Optimality}, we know that there exists $y\in\mathbb{R}^l$ such that
\begin{eqnarray}\label{s1}
0\leq\text{Hess}L(X;y)[\Xi,\Xi]&=&\nabla^2f(X )[\Xi,\Xi] +2\langle\Pi_{\mathrm{N}_{\mathcal{M}^r}(X)}(\nabla_X L(X;y)),\Xi X^{\dag}\Xi \rangle \\\nonumber
&=&\nabla^2f(X )[\Xi,\Xi]-2\langle \nabla_X L(X;y),\Xi X^{\dag}\Xi \rangle,
\end{eqnarray}
where $\Xi\in\mathrm{T}_{\mathcal{L}}(X )\cap\mathrm{T}_{\mathcal{M}^r}(X)=\mathrm{T}_{\mathcal{L}}(X )\cap\mathrm{T}^B_{\mathcal{M}(r)}(X).$

\noindent (ii) Consider the case $s<r$.
For any $\Xi \in \mathrm{T}^B_{{\mathcal{L}}\cap\mathcal{M}(r)}(X )$, there exist $\left\{X^k\right\}\subseteq {\mathcal{L}}\cap {\mathcal{M}}(r)$, $X^k\rightarrow X$ and $t_k\downarrow 0$ such that $\lim\limits_{k\rightarrow \infty} \frac{X^k-X}{t_k} = \Xi$. If $s < r$, then Theorem \ref{F-min} (ii) implies that $X$ is an $F$-stationary point of \ref{P} and hence there exists $y\in \mathbb{R}^l$ such that $\nabla_X L(X;y)= 0$. Thus, we claim that
\begin{equation}\label{gra-vanish}
\langle \nabla_X L(X;y), X^k-X\rangle = 0, ~~\forall k.
\end{equation}
\noindent For any $k$, it then yields that
\begin{eqnarray}
f(X^k) &=& L(X^k;y) \nonumber\\
  &=& L(X;y)+ \frac{1}{2}\nabla_X^2 L(X;y)[X^k-X, X^k-X] + o(\|X^k -X\|^2_F)\nonumber \\
  &=& f(X) + \frac{1}{2}\nabla^2 f(X)[X^k-X, X^k-X] + o(\|X^k -X\|^2_F). \nonumber
\end{eqnarray}
Since $X$ is a local minimizer and $X^k\rightarrow X$, we have $$0\leq \lim\limits_{k\rightarrow \infty} \frac{f(X^k)-f(X)}{t_k^2} = \lim\limits_{k\rightarrow \infty}\frac{1}{2}\nabla^2 f(X)\left[\frac{X^k-X}{t_k}, \frac{X^k-X}{t_k}\right] = \frac{1}{2}\nabla^2 f(X)[\Xi, \Xi].$$
This completes the proof.\end{proof}

\begin{theorem}\label{so-sufficient}
Suppose $f$ is twice continuously differentiable on $\mathbb{R}^{m\times n}$. Let $X\in\mathcal{L}\cap\mathcal{M}(r)$ be an $F$-stationary point of \ref{P} with its SVD as in \eqref{SVD}. Denote $s=\mathrm{rank}(X)$. We have the following statements.
\begin{itemize}
\item[(i)] If $s = r$ and for any  $\Xi\in \left(\mathrm{T}_{\mathcal{L}}(X)\cap\mathrm{T}^B_{\mathcal{M}(r)}(X )\right)\setminus\{O\}$
\begin{equation}\label{semi}
   \nabla^2f(X )[\Xi,\Xi]-2\langle \nabla_X L(X;y),\Xi X^{\dag}\Xi \rangle>0,\end{equation} then $X $ is the strictly local minimizer of \ref{P} restricted on $\mathcal{M}^r$;
\item[(ii)] If $s<r$ and for any  $\Xi\in \left(\mathrm{T}_{\mathcal{L}}(X )\cap\mathrm{T}^B_{\mathcal{M}(r)}(X )\right)\setminus\{O\}$
    \begin{equation}\label{semidef}
     \nabla^2f(X )[\Xi,\Xi]>0, \end{equation}then $X$ is a strictly local minimizer of \ref{P}.
\end{itemize}
\end{theorem}
\begin{proof}
(i) Consider the case $s = r$. According to \eqref{s1}, \eqref{semi} corresponds to the case that the condition $$\text{Hess}L(X;y)[\Xi,\Xi]> 0~~\forall \Xi\in \left(\mathrm{T}_{\mathcal{L}}(X )\cap\mathrm{T}^B_{\mathcal{M}(r)}(X )\right)\setminus\{O\}$$
  automatically holds. Thus, as a result of Theorem 4.3 in \cite{2014Optimality}, $X $ is the strictly local minimizer of \ref{P} restricted on $\mathcal{M}^r$;

(ii)  Consider the case $s<r$. We assume on the contrary that there exists a sequence $\left\{X^k\right\}\subseteq \mathcal{L}\cap\mathcal{M}(r)$ such that $\lim\limits_{k\rightarrow \infty} X^k = X$, $X^k\neq X$, and $f(X^k)\leq f(X)$ for all $k=1, 2, \ldots$. Denote $\Xi^k: = \frac{X^k-X}{\|X^k-X\|_F}$. The boundedness of the sequence $\left\{\Xi^k\right\}$ admits a convergent subsequence. Without loss of generality, we assume that $\Xi^k\rightarrow \Xi$. Thus, $\Xi \in \mathrm{T}^B_{\mathcal{L}\cap\mathcal{M}(r)}(X)$ and $\|\Xi\|_F=1$. Since $s<r$, then the $F$-stationary point $X$ allows us to find some $y\in \mathbb{R}^l$ such that $\nabla_X L(X;y) =0$. It follows readily that
\begin{equation}\label{zero}
\langle \nabla_X L(X;y), X^k-X\rangle = 0, ~~\forall k.
\end{equation}
Direct calculations then yield
\begin{eqnarray*}
0\geq f(X^k)-f(X) &=& L(X^k;y)-L(X;y) \\
&=&\frac{1}{2}\nabla_X^2 L(X;y)[X^k-X, X^k-X]+o\left(\|X^k-X\|^2_F\right),
\end{eqnarray*}
where the first inequality is from the assumption of $f(X^k)\leq f(X)$, the first equality is from the feasibility of $X^k$ and $X$, and the second equality is from \eqref{zero}.
Thus, $$ 0\geq \lim\limits_{k\rightarrow \infty} \frac{f(X^k)-f(X)}{\|X^k-X\|^2_F} = \lim\limits_{k\rightarrow \infty} \frac{1}{2}\nabla_X^2 L(X;y)[\Xi^k, \Xi^k] = \frac{1}{2}\nabla_X^2 L(X;y)[\Xi, \Xi].$$
Note that $\nabla_X^2 L(X;y) = \nabla^2 f(X)$ and $\Xi\in
 \mathrm{T}^B_{\mathcal{L}\cap\mathcal{M}(r)}(X) \subseteq \mathrm{T}_\mathcal{L}(X)\cap \mathrm{T}^B_{\mathcal{M}(r)}(X)$. This arrives at a contradiction to \eqref{semidef}. Thus, $X$ is a strictly local minimizer of \ref{P}.
\end{proof}

\begin{remark}
We have constructed the second-order optimality condition, which gives a supplement to the first-order optimality condition. Particularly, in the case of $s=r$, the second-order sufficient condition proves that the $F$-stationary point is a strictly local minimizer restricted on fixed-rank manifold without any convexity assumption on $f$. For super low-rank cases, this allows one to obtain the strictly local minimizer of \ref{P} by handling a small number (at most $r$) of low fixed-rank manifold optimization problems.
\end{remark}

\section{Applications}
Recently there has been a surge of interest in low-rank matrix optimization subject to some problem-specific constraints often characterized as an affine manifold. In this section, two selected applications are considered, for the purpose of the illustration of our proposed optimality conditions for \ref{P}. 

\subsection{Low-rank Hankel matrix approximation}
Hankel low-rank approximation has appeared in data analysis, system identification, model order reduction, low-order controller design and low-complexity modelling, see, e.g.\cite{fazel2002matrix,fazel2013hankel,Ankelhed,grussler2018low-rank,qi2018a} and references therein. Specifically, in low-order automatic control, the rank of a Hankel matrix is crucial since it reflects the order of a linear dynamical system. Finding a low-rank Hankel matrix approximation can be formulated as
\begin{equation}\label{HMA}
  \begin{aligned}
\min\limits_{X\in\mathbb{R}^{m\times n}}&~~\dfrac{1}{2}\|H-X\|_F^2\\
{\rm s.t.}&~~ X\in\mathcal{H}_{m,n},\\
 &~~ \mathrm{rank}(X)\leq r,
  \end{aligned}
\end{equation}
where
\begin{equation}
{\mathcal{H}}_{m,n}:= \left\{  \left[
  \begin{array}{cccc}
  x_1 & x_2&\cdots & x_n \\
    x_2& x_3& \cdots & x_{n+1} \\
     \vdots&  \vdots& \ddots & \vdots \\
      x_{m}& x_{m+1}& \cdots & x_{m+n-1}\\
     \end{array}
 \right]\in{\mathbb{R}}^{m\times n}:x_i\in {\mathbb{R}}, i=1,\ldots, m+n-1\right\}.
\end{equation}
Set $f(X) := \frac{1}{2}\|H-X\|_F^2$ and $l:=(m-1)(n-1)$. For  $i=1,\ldots, l$, set
$$A^i: = e_ke_j^\top -e_{k-1}e_{j+1}^\top, ~k=1,\ldots, m, ~j=1,\ldots, n, ~~\mbox{and}~~b_i = 0.$$
Then problem \eqref{HMA} turns out to be \ref{P}.
For illustration purpose, we consider a simple case where $m=n=3$, $r=2$, and
$$H=\left[
            \begin{array}{ccc}
              112 & 7.5 & 0 \\
              7.5 & 0 & 0 \\
              0 & 0 & 10^{-6} \\
            \end{array}
 \right].$$
In this case, $l=4$, and the matrices $A^{i}$'s are $$A^1=\left[
 \begin{array}{ccc}
      e_2 & -e_1 & 0 \\
     \end{array}
  \right],~~~~~A^2=\left[
   \begin{array}{ccc}
      0 & e_2 & -e_1 \\
     \end{array}
  \right],$$
  $$A^3=\left[
     \begin{array}{ccc}
    e_3 & -e_2 & 0\\
     \end{array}
  \right],
 ~~~~~A^4=\left[
     \begin{array}{ccc}
      0 & e_3 & -e_2 \\
     \end{array}
  \right].$$
Let us consider the feasible solution $\bar{X}$ with the multiplier vector $\bar{y}$ as follows:
\begin{equation}\label{target}
\bar{X} : = \left[
            \begin{array}{ccc}
              112 & 7.5 & 0 \\
              7.5 & 0 & 0 \\
              0 & 0 & 0 \\
            \end{array}
 \right], ~~\bar{y} = 0\in {\mathbb{R}}^4.
 \end{equation} One can easily get SVD of $\bar{X}=\bar{U}\bar{\Sigma}\bar{V}$ with
 \begin{equation}\label{bar-SVD}
 \bar{U}=\bar{V}=\left[
 \begin{array}{ccc}
      -\sqrt{\frac{112.5}{113}}e_1-\sqrt{\frac{0.5}{113}}e_2 &  \sqrt{\frac{0.5}{113}}e_1-\sqrt{\frac{112.5}{113}}e_2 & e_3\\
     \end{array}
  \right],
 \end{equation}
and $$ \bar{\Sigma} = {\rm Diag} (112.5, -0.5, 0).$$
Let $\bar{\Gamma}:= \{1,2\}$.

By utilizing the first-order and the second-order optimality conditions delivered in Section 4, we can obtain the following assertions.
\begin{itemize}
\item[(i)]  $\bar{X}$ is an $F$-stationary point of problem \eqref{HMA} associated with $\bar{y}$;
\item[(ii)] $\bar{X}$ is a strictly local minimizer of \eqref{HMA} restricted on $\mathcal{M}^2$;
\item[(iii)] $\bar{X}$ is a strictly local minimizer of \eqref{HMA}.
\end{itemize}
\vskip 2mm

{\tt For (i):} Direct manipulations yield
\begin{equation} -\nabla_X L(\bar{X};\bar{y}) = 10^{-6}e_3e_3^\top \in \left\{\alpha e_3e_3^\top:\alpha\in{\mathbb{R}}\right\}=\mathrm{N}^{F}_{\mathcal{M}(2)}(\bar{X}). \end{equation}
By Definition \ref{def-stat}, $\bar{X}$ is an $F$-stationary point of problem \eqref{HMA}.
\vskip 2mm

{\tt For (ii):}
When applying the first-order optimality as stated in Theorem \ref{F-min} (ii), together with the assertion in (i) and the convexity of $f$, one can only obtain that $\bar{X}$ is a global minimizer of \eqref{HMA} restricted on $\mathcal{M}_{\bar{X}}(\bar{\Gamma})$. To get the desired assertion in (ii), we need the second-order optimality conditions. Note that for any nonzero matrix $\Xi\in {\mathbb{R}}^{3\times 3}$, one has

    \begin{equation}\label{App1}
   \nabla^2f(\bar{X})[\Xi,\Xi]-2\langle \nabla_X L(\bar{X};\bar{y}),\Xi \bar{X}^{\dag}\Xi \rangle=\|\Xi\|_F^2-\langle 10^{-6}e_3e_3^\top, \Xi \bar{X}^{\dag}\Xi\rangle>0.
   \end{equation}
Thus, utilizing the second-order sufficient condition as stated in Theorem \ref{so-sufficient} (i), we can obtain that $\bar{X}$ is a strictly local minimizer of \eqref{HMA} restricted on $\mathcal{M}^2$.

\vskip 2mm

{\tt For (iii):} With the optimality of $\bar X$ in (ii) just proved, combining with the fact ${\mathcal{M}}(2)= {\mathcal{M}}^2\cup {\mathcal{M}}(1)$, it suffices to show that for any global minimizer, namely $\tilde{X}$, of the following problem
\begin{equation}\label{M1}
\min_{X\in {\mathbb{R}}^{3\times 3}} \{f(X): X\in {\mathcal{F}}_1:= {\mathcal{H}}_{3,3}\cap {\mathcal{M}}(1)\},
\end{equation}
one has $f(\bar X)<f(\tilde{X})$.  Observe that
$${\mathcal{F}}_1 = \left\{t_1e_1e_1^\top:t_1\in {\mathbb{R}}\right\} \cup \left\{t_2ee^\top:t_2\in {\mathbb{R}}\right\} \cup \left\{t_3e_3e_3^\top:t_3\in {\mathbb{R}}\right\},$$
where $e$ is the all-one vector. Direct manipulations yield that $\tilde{X} = 112 e_1e_1^\top$. Obviously,
$f(\bar X)<f(\tilde{X})$. Thus, $\bar{X}$ is a strictly local minimizer of \eqref{HMA}.

\vskip 2mm

There are two additional things which are noteworthy.
\begin{itemize}
\item By the optimality addressed in (iii), one can conversely verify that (i) holds by applying the first-order optimality condition presented in Theorem \ref{F-min} (i), since the required Assumption \ref{assumption 1} holds at $\bar{X}$. Specifically, one can check that $$T_{\bar{X}}^1=\left[
 \begin{array}{ccc}
      0 & -1 & 0 \\
        1 & 0& 0 \\
          0 & 0& 0 \\
     \end{array}
  \right],~~~~~T_{\bar{X}}^2=\left[
   \begin{array}{ccc}
      b^2 & ab & a \\
          ab & a^2 & -b \\
             0& 0 & 0 \\
     \end{array}
  \right],$$
  $$T_{\bar{X}}^3=\left[
     \begin{array}{ccc}
   -b^2 & -ab & 0 \\
          -ab & -a^2 & 0 \\
             -a& b & 0 \\
     \end{array}
  \right],
 ~~~~~T_{\bar{X}}^4=\left[
     \begin{array}{ccc}
      0 & 0 & b \\
          0 & 0 & a \\
             -b& a & 0 \\
     \end{array}
  \right]$$
 with $a=\sqrt{\dfrac{112.5}{113}}$ and $b=\sqrt{\dfrac{0.5}{113}}$,  are linearly independent.
\item Note that $\tilde{X}$ is also the unique global minimizer of problem \eqref{M1}. However, $\tilde{X}$ is not an $F$-stationary of \eqref{M1}, since for any $y\in\mathbb{R}^4$,
\begin{eqnarray}
 -\nabla_X L(\tilde{X}; y) &=& \left[
      \begin{array}{ccc}
   0 & 7.5+y_1 & y_2  \\
   7.5-y_1 & y_3-y_2 & y_4\\
    -y_3 & -y_4 & 10^{-6} \\
     \end{array}
   \right]\nonumber\\
   &\notin & \left\{ \left[
      \begin{array}{ccc}
   0 &  0& 0 \\
   0 &  a_1& a_2 \\
   0 &  a_3 & a_4  \\
     \end{array}
   \right]:a_i\in {\mathbb{R}}, i=1,\ldots,4\right\}\nonumber\\
   &=&\mathrm{N}^{F}_{\mathcal{M}(1)}(\tilde{X}).
   \end{eqnarray}
The reason for this is the failure of Assumption \ref{assumption 1} at $\tilde{X}$, as one can see that $T_{\tilde X}^4=O$. It indicates that the Assumption \ref{assumption 1} in the first-order optimality in Theorem \ref{F-min} (i) cannot be removed in general.

\end{itemize}

\subsection{Low-Rank representation over the manifold}
Low-rank representation (LRR) has recently attracted considerable interest as its pleasing efficacy in exploring low-dimensional subspace structures embedded in data, which is very helpful for data clustering. However, in many computer vision applications, data often originate from a manifold, which is equipped with some Riemannian geometry, and the low-rank representation over the manifold \cite{Yin2015Nonlinear, YFu2105, Wang2015Kernelized} is required. This problem can be formulated as
\begin{equation}\label{cLRR}
  \begin{aligned}
\min\limits_{W\in\mathbb{R}^{N\times N}}&~~\dfrac{1}{2}\sum_{i=1}^N w_iB^iw_i^\top\\
{\rm s.t.}&~~ \sum_{j=1}^N W_{ij}=1, ~i=1,\ldots, N,\\
 &~~ \mathrm{rank}(W)\leq r,
  \end{aligned}
\end{equation}
where $B^i\in\mathbb{R}^{N\times N}$, $w_i$ is the $i$-th row of matrix $W\in\mathbb{R}^{N\times N}$. Set $$f(W):= \dfrac{1}{2}\sum_{i=1}^N w_iB^iw_i^\top, ~\mbox{ and~} A^i: = E^i\in\mathbb{R}^{N\times N}, ~i=1,\ldots, N$$ with $E^i$ the matrix having all components in the $i$-th row $1$  and others $0$. Thus, problem \eqref{cLRR} is a special case of \ref{P}. The Lagrangian function of \eqref{cLRR} is given by
$$L(W;y)=\dfrac{1}{2}\sum_{i=1}^N w_iB^iw_i^\top+\sum_{i=1}^N y_i[\langle E^i, W\rangle -1]$$
where $y=(y_1,\ldots,y_N)^\top $ is the Lagrangian multiplier vector corresponding to the equality constraint.
One has
$$\nabla_WL(W;y)=\left[
     \begin{array}{ccc}
      w_1B^1 \\
      \vdots\\
     w_NB^N \\
     \end{array}
  \right]+ye^\top.$$
For illustration, we simply take $B^i=I_N$ for $i=1,\ldots, N$ and $r>1$ as an example. Clearly, the objective function is convex, and the gradient
$\nabla_WL(W;y)=W+y e^\top.$
Let us consider the feasible solution $\overline{W}$ with the multiplier vector $\bar{y}$ as follows:
 $$\overline{W}=\dfrac{1}{N}ee^\top, ~~\bar{y}=-\dfrac{1}{N}e\in {\mathbb{R}}^N.$$
One can prove that $\overline{W}$ is a global minimizer of \eqref{cLRR} by applying the first-order and the second-order optimality conditions as proposed in Section 4. As a start, one can easily obtain that $\overline{W}$ is an $F$-stationary point of problem due to the fact $$\nabla_WL(\overline{W};\bar{y})=O.$$
Since $\text{rank}(\overline{W})=1<r$, it then follows from the first-order optimality condition in Theorem \ref{F-min} (ii) that $\overline{W}$ is a global minimizer of \eqref{cLRR}. Furthermore, one can  use the second-order sufficient condition in Theorem \ref{so-sufficient} (ii) to show that $\overline{W}$ is also the unique global minimizer, since
$$\nabla^2f(\overline{W})[\Xi,\Xi]=\|\Xi\|_F^2>0, ~~\forall \Xi\neq O.$$ 

Conversely, let the SVD of $\overline{W}$ be $\overline{W} = U \Sigma V^\top$ with $\Sigma = {\rm Diag}(1, 0,\ldots, 0)$, $U=V$ whose first column is $\dfrac{1}{\sqrt{N}}e$. Let $\Gamma =\{1\}$. Direct calculations yield
$$R^i_{\overline{W}} = \dfrac{1}{\sqrt{N}} U^\top e_i, ~i = 1,\ldots, N.$$
Thus, Assumption \ref{assumption 2} holds at $\overline{W}$. In this case, under the global optimality of $\overline{W}$ to problem \eqref{cLRR}, the first-order optimality condition as discussed in Theorem \ref{F-min} (i) yields that $\overline{W}$ is an $F$-stationary point of  problem \eqref{cLRR}.

\section{Conclusions}
The nonlinear matrix optimization constrained by the low-rank matrix set intersecting with an affine manifold, termed as \ref{P}, has been studied in this paper, emphasizing on the first-order and the second-order optimality conditions.
We have explored the intersection rule of Fr{\'e}chet normal cone to the underlying feasible set relying on two linear independence assumptions for the cases of $s<r$ and $s=r$, respectively. This further has allowed us to derive the first-order necessary and sufficient optimality conditions for the \ref{P} via the $F$- and the $\alpha$-stationarity. Moreover, the second-order necessary and sufficient optimality condition are also presented based on the Bouligand tangent cone. To illustrate the results of these optimality conditions, two specific applications of \ref{P} are discussed. To the best of our knowledge, this paper is the first one to touch the optimality conditions for the original low-rank optimization problem \ref{P}.

It is worth mentioning the $\alpha$-stationary point, defined via the tractable low-rank matrix projection, might provide clues for algorithm design. Existing related work can be found in \cite[Theorem 3.4]{Jain2014} for \ref{P} with vacant $\mathcal{L}$, where the projected gradient descent algorithm is designed with the iteration scheme
$$ X_{k+1} \in\Pi_{\mathcal{M}(r)}(X_k -\alpha_k\nabla f(X_k)).$$
For the general case with the affine manifold in \ref{P}, a possible way for algorithm design would be working with the nonlinear system \eqref{alpha-stationary-def} in the definition of the $\alpha$-stationary point. One can also find the vector counterpart in sparse optimization in \cite{Zhao2021}, where the Lagrange-Newton algorithm was proposed and showed to possess quadratic convergence. For \ref{P}, and even for more general cases including additional nonlinear equality and inequality constraints, how to design efficient methods based on the stationarity deserves further investigation.

\bibliographystyle{plain}
\bibliography{sdp}

\end{document}